\documentclass[11pt]{amsart}  

\usepackage[utf8]{inputenc}
\usepackage{amsmath}
\usepackage{amsfonts}
\usepackage{amssymb}
\usepackage{latexsym}
\usepackage{graphicx}
\usepackage{subfigure}
\usepackage{color}
\usepackage{mathtools}
\usepackage[backref=page]{hyperref}
\usepackage{verbatim}
\usepackage[all]{xy}
\usepackage{pdfsync}
\usepackage{youngtab}
\usepackage{ytableau}
\usepackage{young}
\usepackage{tikz}
\usepackage{cancel} 
\usepackage[normalem]{ulem} 
\usepackage{graphicx}
\usepackage{forest}
\usepackage{array}
\usepackage{extarrows}

\theoremstyle{plain}
\newtheorem{theorem}{Theorem}[section]

\newtheorem{lemma}[theorem]{Lemma}
\newtheorem{corollary}[theorem]{Corollary}
\newtheorem{conjecture}[theorem]{Conjecture}

\newtheorem{example}[theorem]{Example}

\numberwithin{equation}{section}
\definecolor{emphcol}{rgb}{0.0, 0.5, 2.0}

\newcommand{\Z}{{\mathbb Z}}
\newcommand{\N}{{\mathbb N}}
\newcommand{\C}{{\mathbb C}}
\DeclareMathOperator{\GCD}{GCD}
\DeclareMathOperator{\PF}{PF}

\DeclareMathOperator{\md}{mod}
\DeclareMathOperator{\cat}{Cat}

\newcommand{\epf}[1]{\widehat{\mathrm{PF}}_{#1}} 
\newcommand{\bfx}{\mathbf{x}} 
\newcommand{\area}[1]{\mathrm{area}(#1)} 
\begin{document}

\title[]{Some natural extensions of the parking space}
\author{Matja\v z Konvalinka}
\address{Department of Mathematics, University of Ljubljana \& Institute of Mathematics, Physics and Mechanics, Ljubljana, Slovenia}
\email{\href{mailto:matjaz.konvalinka@fmf.uni-lj.si}{matjaz.konvalinka@fmf.uni-lj.si}}
\author{Vasu Tewari}
\address{Department of Mathematics, University of Pennsylvania, Philadelphia, PA 19104, USA}
\thanks{The first author acknowledges the financial support from the Slovenian Research Agency (research core funding No. P1-0294).}
\email{\href{mailto:vvtewari@math.upenn.edu}{vvtewari@math.upenn.edu}}

\begin{abstract}
We construct a  family of $S_n$ modules indexed by $c\in\{1,\dots,n\}$ with the property that upon restriction to $S_{n-1}$ they recover the classical parking function representation of Haiman.
The construction of these modules relies on an $S_n$-action on a set that is closely related to the set of parking functions.
We compute the characters of these modules and use the resulting description to classify them up to isomorphism.
In particular, we show that the number of isomorphism classes is equal to the number of divisors $d$ of $n$ satisfying $ d\neq 2 \: (\!\!\!\!\mod 4)$.
In the cases $c=n$ and $c=1$, we compute the number of orbits.
Based on empirical evidence, we conjecture that when $c=1$, our representation is $h$-positive and is in fact the (ungraded) extension of the parking function representation constructed by Berget and Rhoades.
\end{abstract}

\maketitle

\section{Introduction}

Parking functions were introduced by Konheim and Weiss \cite{KW68} in their investigation of hashing functions in computer science.
Since then, they, along with their various generalizations, have attracted plenty of attention and have proven to be a fertile source of interesting mathematics. This is reflected by their appearances in diverse areas such as hyperplane arrangements \cite{AL99,PH12,Maz17,St96}, representation theory \cite{PP94, ARR15},  polytopes \cite{SP02}, the sandpile model \cite{CLB03}, and the theory of Macdonald polynomials \cite{Hai94}.
The last of these areas provides the context for our work and we detail our motivation next.

\medskip

An integer sequence $(x_1,\ldots,x_n)$ is a \emph{parking function} if its weakly increasing rearrangement $(z_1,\ldots,z_n)$ satisfies $0 \leq z_i \leq i - 1$ for $i = 1,\ldots,n$.
This definition implies that rearranging the entries in one parking function results in another.
Haiman \cite{Hai94} was the first to study the $S_n$ action on the set of parking functions of length $n$.
We denote the resulting $S_n$-representation by $\rho_n$.
Two decades later, Berget-Rhoades \cite{BR14}
studied the following seemingly unrelated representation $\sigma_n$ of $S_n$.
Let $K_n$ denote the complete graph with vertex set $[n]\coloneqq \{1,\dots, n\}$.
Given a subgraph $G\subseteq K_n$, we attach to it the polynomial  $p(G) \coloneqq \prod_{ij \in E(G)} (x_i - x_j) \in \C[x_1,\ldots,x_n]$.
Here $E(G)$ refers to the set of edges of $G$ and we record those by listing the smaller number first.
Define $V_n$ to be the $\mathbb{C}$-linear span of $p(G)$ over all $G$ for which the complement $\overline G$ is a connected graph.
We remark here that $V_n$ first appears in the work of Postnikov and Shapiro \cite{Pos04}, where the graphs $G$ with the property that $\overline G$ is connected are called \emph{slim graphs}.
The natural action of $S_n$ on $\C[x_1,\ldots,x_n]$ that permutes variables gives an action on $V_n$ because relabeling vertices preserves connectedness.
Amongst various other interesting things, Berget and Rhoades \cite[Theorem 2]{BR14} establish the remarkable fact that the restriction of $\sigma_n$ to $S_{n-1}$ is isomorphic to $\rho_{n-1}$.
The question of extending symmetric group representations in general has also received attention; see \cite{Ma96,Sun01}.

\medskip

The primary goal of this article is to construct a family of permutation representations $\epf{n,c}$ of $S_n$ with easy-to-compute characters, which all also restrict to $\rho_{n-1}$.
Interestingly, the modular behavior of the sum of elements in a parking function (closely related to the area statistic on parking functions) plays a key role in our analysis, {and our arguments rely on some subtle number-theoretic considerations.}
The authors in fact believe that the representation $\epf{n,1}$ is isomorphic to the (ungraded) Berget-Rhoades representation mentioned above; see Conjecture \ref{conj1}.

\section{Background}
For any undefined terminology in the context of symmetric functions, we refer the reader to \cite{St99,Mac95}.
For $n\geq 1$, we denote by $\Z_n$ the set of integers modulo $n$.
Typically, representatives from residue classes modulo $n$ will be implicitly assumed to belong to $\{0,\dots,n-1\}$.
Throughout, $S_n$ denotes the symmetric group consisting of permutations of $[n]$.
We use both the cycle notation and the one-line notation for permutations depending on our needs.
If we use the latter, then we let $\pi_i$ denote the image of $i$ under the permutation $\pi$ for a positive integer $i$.

\subsection{Symmetric functions}
A \emph{partition} $\lambda = (\lambda_1, \dots, \lambda_{\ell})$ is a weakly decreasing sequence of positive integers. The $\lambda_i$'s' are  the \emph{parts} of $\lambda$, their sum its \emph{size}, and their number its \emph{length}, which is denoted by $\ell(\lambda)$.
If $\lambda$ has size $n$, then we denote this by $\lambda\vdash n$.
Furthermore, letting $m_i$ denote the multiplicity of the part $i$ in $\lambda$ for $i\geq 1$, we set $z_{\lambda}\coloneqq \prod_{i\geq 1}i^{m_i}m_i!$. The cycle type of a permutation $\pi$ is a partition that we denote $\lambda(\pi)$.

\medskip

We consider the following distinguished bases for the ring of symmetric functions $\Lambda$: the \emph{power sum} symmetric functions $\{p_{\lambda}\colon \lambda\vdash n\}$, the \emph{complete homogeneous} symmetric functions $\{h_{\lambda}\colon \lambda\vdash n\}$, and the \emph{Schur} symmetric functions $\{s_{\lambda}\colon \lambda\vdash n\}$.

\medskip

The representation theory of the symmetric group is intimately tied to $\Lambda$ and the connection is made explicit by the \emph{Frobenius characteristic}. Given a representation $\rho$ of $S_n$, denote the corresponding character by $\chi_{\rho}$. Then
\[
\mathrm{Frob}(\rho)= \frac{1}{n!} \sum_{\pi \in S_n} \chi_\rho(\pi) p_{\lambda(\pi)} =  \sum_{\lambda\vdash n}\chi_{\rho}(\lambda)\frac{p_{\lambda}}{z_{\lambda}}.
\]
Under $\mathrm{Frob}$, the irreducible representation of $S_n$ corresponding to the partition $\mu\vdash n$ gets mapped to the Schur function $s_{\mu}$.
As a special case, we have the equality $\sum_{\lambda \vdash n} z_\lambda^{-1} p_\lambda = h_n$.
We proceed to define parking functions and an associated representation whose study has substantially motivated algebraic combinatorics in the last two decades.

\subsection{Parking functions}
As mentioned earlier, an integer sequence $(x_1,\ldots,x_n)$ is a parking function if its weakly increasing rearrangement $(z_1,\ldots,z_n)$ satisfies $0 \leq z_i \leq i - 1$ for $i = 1,\ldots,n$.
We denote by $\PF_n$ the set of all parking functions of length $n$.
For example,
\begin{align*}
\PF_2 =& \{00,01,10\}, \\
\PF_3 =& \{000,001,010,100,002,020,200,011,\\
& 101,110,012,021,102,120,201,210 \},
\end{align*}
and the weakly increasing elements of $\PF_4$ are $0000$, $0001$, $0011$, $0111$, $0002$, $0012$, $0112$, $0022$, $0122$, $0003$, $0013$, $0113$, $0023$, and $0123$.
Observe that there are 14 such elements in $\PF_4$.
More generally, we have that the number of weakly increasing elements in $\PF_n$ is the $n$th Catalan number $\mathrm{Cat}_n\coloneqq \frac{1}{n+1}\binom{2n}{n}$.
In the preceding examples, we have omitted commas and parentheses in writing our parking functions for the sake of clarity, and we will do this throughout without explicit mention.

\medskip

It is well known that $|\PF_n| = (n+1)^{n-1}$.
One way to see this is through the following result present in \cite{FR74} (where it is attributed to H.~O.~Pollak)  that will also be crucial in the sequel.

\begin{theorem}[Pollak]
	The map $\PF_{n} \to \Z_{n+1}^{n-1}$, given by
	$$(x_1,\ldots,x_{n}) \mapsto (x_2-x_1,\ldots,x_{n}-x_{n-1}),$$
	where subtraction is performed modulo $n+1$, is a bijection.
	\label{thm:Pollak}
\end{theorem}

\noindent Note that in particular Theorem~\ref{thm:Pollak} says that for an arbitrary sequence $(\alpha_1,\ldots,\alpha_{n-1}) \in \Z_{n+1}^{n-1}$, exactly one of the sequences $(y,y+\alpha_1,y+\alpha_1+\alpha_2,\ldots,y+\alpha_1+\cdots+\alpha_{n-1})$, $y \in \Z_{n+1}$, is in $\PF_{n}$.

\medskip

Recall the natural action $\rho_{n}$ of  $S_{n}$ on $\PF_{n}$ defined by
$$\pi \cdot (x_1,\ldots,x_{n}) = (x_{\pi_1},\ldots,x_{\pi_{n}}).$$
For a partition $\lambda = (\lambda_1,\ldots \lambda_{\ell}) \vdash n$, the number of fixed points of the action of the permutation with cycle decomposition $(1, \ldots, \lambda_1)(\lambda_1+1, \ldots, \lambda_1+\lambda_2) \cdots$ is equal to the number of sequences $(\alpha_1,\ldots,\alpha_{n-1}) \in \Z_{n+1}^{n-1}$ satisfying $\alpha_i = 0$ for $i \in [n-1] \setminus \{\lambda_1,\lambda_1+\lambda_2,\ldots,\lambda_1+\cdots+\lambda_{\ell-1}\}$.
It follows that the character $\chi_{\rho_{n}}$ of $\rho_{n}$ satisfies
$$\chi_{\rho_{n}}(\pi) = (n+1)^{\ell - 1},$$
where $\ell\coloneqq \ell(\lambda(\pi))$.


\section{Main results}\label{sec:main_res}
For $n \in \N$ and $1\leq c\leq n$, define the set
$$\epf{n,c} = \{(x_1,\dots, x_n)\in \Z_{n}^{n} \colon (x_1,\ldots,x_{n-1}) \in \PF_{n-1}, \: x_n = c-\sum_{1\leq i\leq n-1}x_i \: (\md n) \}.$$
In other words, given a parking function $(x_1,\ldots,x_{n-1})$, $x_n$ is uniquely determined by the constraint $\sum_{i=1}^n x_i = c \: (\md n)$.
For example, the reader may check that
\begin{align*}
\epf{3,1} = & \{001,010,100\}, \\
\epf{3,2} = & \{002,011,101\}, \\
\epf{3,3} = & \{000,012,102\}.
\end{align*}

\medskip

It is obvious that, for every $1\leq c\leq n$, the projection $(x_1,\ldots,x_n) \mapsto (x_1,\ldots,x_{n-1})$ is a bijection $\epf{n,c} \to \PF_{n-1}$.
In particular, we have $|\epf{n,c}| = n^{n-2}$.
Again, we can construct an action $\tau_{n,c}$ of  $S_n$ on $\epf{n,c}$.
Take $\pi \in S_n$ and $(x_1,\ldots,x_n) \in \epf{n,c}$.
Note that $(x_{\pi_1},\ldots,x_{\pi_{n-1}})$ is not necessarily in $\PF_{n-1}$, and therefore $(x_{\pi_1},\ldots,x_{\pi_n})$ is not necessarily in $\epf{n,c}$.
However, by Pollak's theorem, exactly one of the sequences $(y+x_{\pi_1},\ldots,y+x_{\pi_{n-1}})$ is in $\PF_{n-1}$, and therefore $(y+x_{\pi_1},\ldots,y+x_{\pi_{n}}) \in \epf{n,c}$.
This element is the action of $\pi$ on $(x_1,\ldots,x_n)$. For example, consider the action of $\pi=1432\in S_4$ on $0003\in \epf{4,3}$. Na\"ively permuting elements of the sequence $0003$ according to $\pi$ leads to $0300$. Note that $030\notin \PF_3$, but adding $1$ to each coordinate gives $101\in \PF_3$.
Thus $1432 \cdot 0003 = 1011$.

\medskip

The following is our first main result.
\begin{theorem}
  \label{thm:main_1}
  The map $\tau_{n,c}$ is an action of $S_n$ on $\epf{n,c}$ whose restriction to $S_{n-1}$ is isomorphic to $\rho_{n-1}$.
  Furthermore, the character $\chi_{n,c}\coloneqq\chi_{\tau_{n,c}}$ can be computed as follows.
  Choose a permutation $\pi \in S_n$ with cycle type $\lambda=(\lambda_1,\dots,\lambda_\ell)$, and write $d \coloneqq \GCD(\lambda_1,\ldots,\lambda_\ell)$. Then
  $$
  \chi_{n,c}(\pi)=\left\lbrace \begin{array}{ll}\frac{d^2n^{\ell-2}}{2} & d \text{ even,} \frac{n}{d} \text{ odd, and } d|2c \\
  d^2n^{\ell-2} & d \text{ even,} \frac{n}{d} \text{ even, and } d|c \\
  d^2n^{\ell-2} & d \text{ odd and }  d|c\\
  0 & \text{otherwise.}
  \end{array}\right.
  $$
\end{theorem}
As a corollary, we completely classify the representations $\tau_{n,c}$ up to isomorphism, and show in particular that the number of non-isomorphic representations is equal to the number of divisors of $n$ that are not $2 \: (\md 4)$.
We refer the reader to Section~\ref{sec:understanding_pfnc} for further details, in particular to Theorem~\ref{thm:iso_classes} and Corollary~\ref{cor:num_noniso_classes}.

\medskip

Subsequently we focus on the cases where $c$ equals $n$ (equivalently, $0$) and $1$, where the characters simplify a bit. In both cases we compute the multiplicity of the trivial representation in $\epf{n,c}$, or equivalently, the number of orbits  under $\tau_{n,c}$.
 As our second main result, we state below the character in the case $c=1$ as well as the number of orbits.

\begin{theorem} \label{thm2}
   The character $\chi_{n,1}$ can be computed as follows.
    Choose a permutation $\pi \in S_n$ with cycle type $\lambda=(\lambda_1,\dots,\lambda_\ell)$, and write $d \coloneqq \GCD(\lambda_1,\ldots,\lambda_\ell)$. Then
    $$\chi_{n,1}(\pi) = \begin{cases} n^{\ell-2} & d = 1 \\ 2n^{\ell-2} & d = 2, \: n = 2 \: (\md 4) \\ 0 & \text{otherwise.} \end{cases}.$$
  As a consequence, the number of orbits of the action $\tau_{n,1}$ is given by
 $$o_{n,1} = \frac{1}{n^2} \sum_{d | n}(-1)^{n+d} \mu(n/d) \binom{2d-1}d,$$
 where $\mu$ is the classical M{\"o}bius function.
\end{theorem}
\noindent Note that the sequence $(o_{n,1})_{n \in \N}$ starts with $1, 1, 1, 2, 5, 13, 35, 100, 300$ (see \cite[A131868]{oeis}).

\medskip

Recall from the introduction that understanding the Berget-Rhoades extension was our main motivation. In this context, we  offer the following conjecture to close this section.
\begin{conjecture} \label{conj1}
  The representation $\tau_{n,1}$ is isomorphic to $\sigma_n$. Furthermore, $\mathrm{Frob}(\tau_{n,1})$ expands positively in the basis of homogeneous symmetric functions, i.e., it is \emph{$h$-positive}.
\end{conjecture}
\noindent It is worth noting that from the original definition of $\sigma_n$ in terms of slim graphs, it is not straightforward to compute its character.
In this regard, assuming the validity of Conjecture~\ref{conj1}, one could say that $\tau_{n,1}$ is the computationally more amenable representation.

\section{Characters and classification of the \texorpdfstring{$\epf{n,c}$}{}}
\label{sec:understanding_pfnc}
Before providing proofs to our main results stated earlier, we establish a useful lemma.
\begin{lemma} \label{lem:modular_lemma}
	For $a_1,\ldots,a_k,c \in \Z$, $m \in \N$ the number of tuples $(x_1,\dots,x_k)\in \{0,\dots,m-1\}^{k}$ that satisfy
	$$a_1 x_1 + \cdots + a_k x_k  = c \: (\md m)$$
	is equal to $d m^{k-1}$ if  $d | c$, and $0$ otherwise.
	Here  $d = \GCD(a_1,\ldots,a_k,m)$.
\end{lemma}
\begin{proof}
	Consider the homomorphism from the group $\mathbb{Z}_m^k$ to $\mathbb{Z}_m$ sending $(x_1,\dots,x_k)$ to $a_1x_1+\dots+a_kx_k \: (\md m)$.
	It follows from the extended Euclidean algorithm that the image is the set $\{y\in \Z_m \colon d|y\}$.
	Thus, we see that there exists a solution to the equation in question if $d|c$. Furthermore, if this is indeed the case, the number of solutions is equal to the cardinality of the kernel of our map, i.e., $\frac{m^k}{m/d}=dm^{k-1}$, thereby implying the claim.
\end{proof}

We are  ready to compute the character $\chi_{n,c}$ of the $S_n$ action on $\epf{n,c}$.
\subsection*{Proof of Theorem~\ref{thm:main_1}}
Since the maps $(x_1,\ldots,x_n) \mapsto (x_{\pi_1},\ldots,x_{\pi_n})$ and $(x_1,\ldots,x_n) \mapsto (y+x_1,\ldots,y+x_n)$ commute, we conclude that $\tau_{n,c}$ is an action.
It is also clear that the restriction of $\tau_n$ to $S_{n-1}$ is $\rho_{n-1}$.
It remains to compute the character $\chi_{n,c}$.

\medskip

Without loss of generality, assume that $\pi = (1, \ldots, \lambda_1)(\lambda_1+1, \ldots, \lambda_1+\lambda_2) \cdots$, and set $d \coloneqq \GCD(\lambda)$.
Also, following \cite[Equation 7.103]{St99}, define
\begin{align}
  b(\lambda')\coloneqq \sum_{i = 1}^\ell\binom{\lambda_i}{2}
\end{align}
where $\lambda'$ denotes the transpose of $\lambda$.
As $\lambda$ is fixed, we set $b\coloneqq b(\lambda')$ for convenience. We want to count the number of fixed points of $\pi$.

\medskip

Suppose that $\pi \cdot (x_1,\ldots,x_n) = (x_1,\ldots,x_n)$.
We have $\pi \cdot (x_1,\ldots,x_n) = (x_2+y,\ldots,x_{\lambda_1}+y,x_1+y,x_{\lambda_1+2}+y,\ldots,x_{\lambda_1+\lambda_2}+y,x_{\lambda_1+1}+y,\ldots )$ for some $y \in \Z_n$, so $x_1 = x_2 + y, x_2 = x_3 + y, \ldots, x_{\lambda_1-1} = x_{\lambda_1} + y, x_{\lambda_1} = x_1+y$, $x_{\lambda_1+1} = x_{\lambda_1+2}+y, x_{\lambda_1+2} = x_{\lambda_1+3}+y,\ldots x_{\lambda_1+\lambda_2-1} = x_{\lambda_1+\lambda_2}+y, x_{\lambda_1+\lambda_2} = x_{\lambda_1+1}+y$ etc.

\medskip

The equalities immediately imply that $\lambda_i \cdot y = 0 \: (\md n)$, and consequently $d \cdot y = 0 \: (\md n)$.
In other words, $y = k \cdot \frac n d$ for some $k \in \Z$, $0 \leq k < d$.
Furthermore, the sum of the coordinates of $\pi \cdot (x_1,\ldots,x_n)$ is, modulo $n$, equal to $c$, and therefore
\begin{align}
\lambda_1 x_{1} + \binom{\lambda_1} 2 y + \lambda_2 x_{\lambda_1+1} + \binom{\lambda_2} 2 y + \cdots + \lambda_\ell x_{\lambda_1+\cdots+\lambda_{\ell-1}+1} + \binom{\lambda_\ell} 2 y= c \: (\md n).
\end{align}
Set $f_1\coloneqq x_{1}$, $f_2\coloneqq x_{\lambda_1+1},\dots$, $f_{\ell}=x_{\lambda_1+\dots+\lambda_{\ell-1}+1 }$.
Then counting fixed points of $\pi$ is tantamount to counting tuples $(f_1,\dots,f_l)\in \Z_n^{\ell}$ (up to translation by $(1,\dots,1)\in\Z_n^{\ell}$)  that satisfy
\begin{align}\label{eqn:main_lin_diophantine}
  \sum_{i = 1}^\ell\lambda_if_i+yb = c \: (\md n).
\end{align}

\medskip

Assume first that $d$ is odd.
Then $d|\lambda_i$ implies $d|\binom{\lambda_i} 2$,  and therefore $d|b$.
It follows that
$$yb =   \frac{b}{d}\cdot k \cdot n = 0 \: (\md n),$$
 which in turn implies that \eqref{eqn:main_lin_diophantine} reduces to
\begin{align}
  \sum_{i = 1}^\ell\lambda_if_i = c \: (\md n).
\end{align}
 Using Lemma~\ref{lem:modular_lemma} and recalling that  we have $d$ choices for $y$, we infer that there are $d^2n^{\ell-2}$ (we have power of $\ell-2$ instead of $\ell-1$ because we look at tuples up to translation by $(1,\ldots,1)$, i.e.\, we can fix one of $f_i$'s to be, say, $0$) elements in $\epf{n,c}$ fixed by $\pi$ if $d|c$, and $0$ otherwise.

Now assume that $d$ (and consequently $n$) is even.
Then $\frac d 2 |\frac{\lambda_i}2$ and $\frac d 2 | \binom{\lambda_i}2$, thereby implying $d|2b$.
It follows that
\begin{align}\label{eqn:parity of 2b/d}
yb  = \frac{2b}{d} \cdot k  \cdot \frac n 2.
\end{align}
We are naturally led to consider two scenarios based on the parity of $2b/d$. First note that $n/d = \lambda_1/d + \cdots + \lambda_\ell/d$ is odd if and only if the number of odd numbers among $\lambda_1/d,\ldots,\lambda_\ell/d$ is odd. On the other hand $2b/d = \lambda_1(\lambda_1-1)/d + \cdots +  \lambda_\ell(\lambda_\ell-1)/d$, and $\lambda_1-1,\ldots,\lambda_\ell-1$ are all odd, so $2b/d$ is also odd if and only if the number of odd numbers among $\lambda_1/d,\ldots,\lambda_\ell/d$ is odd. In other words, $2b/d$ and $n/d$ have the same parity.

\medskip

Suppose that $2b/d$ and $n/d$ are even.
In view of the equality in \eqref{eqn:parity of 2b/d}, we may rewrite \eqref{eqn:main_lin_diophantine} as
\begin{align}
  \sum_{i = 1}^\ell\lambda_if_i =c \: (\md n).
\end{align}
Like before, we infer that $d^2n^{\ell-2}$ elements in $\epf{n,c}$ are fixed by $\pi$ if $d|c$, and $0$ otherwise.

Finally consider the case where $2b/d$ and $n/d$ are odd.
We need to count solutions to
\begin{align}
  \sum_{i = 1}^\ell\lambda_if_i= c+\frac{kn}{2} \: (\md n).
\end{align}
Note crucially that since $\frac{n}{d}$ is odd, it cannot be that $d$ divides both $c$ and $c+\frac{n}{2}$.
From the odd $k\in \{0,\dots,d-1\}$, we get a contribution of $\frac{d^2n^{\ell-2}}{2}$ if $d|(c+\frac{n}{2})$, and $0$ otherwise.
From the even $k\in \{0,\dots,d-1\}$, we get a contribution of $\frac{d^2n^{\ell-2}}{2}$ if $d|c$, and $0$ otherwise.
We leave it to the reader to check that in the case under consideration we have
\begin{align}
  d|c \text{ or } d|(c+\frac{n}{2}) \Leftrightarrow d|2c.
\end{align}
This concludes our proof.

\subsection{Number of non-isomorphic \texorpdfstring{$\widehat{\mathrm{PF}}_{n,c}$}{}}
Given a positive integer $n$, let $v_2(n)$ denote the $2$-adic valuation of $n$, i.e., the highest power of $2$ that divides $n$.
Define $D_n$ to be the following subset of the set of divisors of $n$:
\begin{align}
  D_n\coloneqq \{k|n :n/k= n \: (\md 2)\}.
\end{align}
For instance, we have $D_{12}=\{1,2,3,6\}$.
We will show that $D_n$ indexes the isomorphism classes of the representations $\tau_{n,c}$.
Prior to that we establish a straightforward lemma on the cardinality of $D_n$.

\medskip

\begin{lemma}
  \label{lem:num_D_n}
  The cardinality of $D_n$ equals the number of divisors of $n$ that are not $2$ modulo $4$.
\end{lemma}
\begin{proof}
Let $d(n)$ denote the number of divisors of $n$.
Then $|D_n|$ equals $d(n)$ if $v_2(n)=0$, and $v_2(n)\cdot d\left(\frac{n}{2^{v_2(n)}}\right)$ otherwise.
It is easily checked the number of divisors of $n$ that are not $2$ modulo $4$ satisfies the same recursion: such a divisor $d$ must satisfy $v_2(d)\neq 1$.
\end{proof}

\medskip

For $k\in D_n$, consider the set
 \begin{align}
  C_{n,k}\coloneqq \{ m\in [n]:  \GCD(n,m)\in \{k,2k\} \text{ if } \frac{n}{k}= 2 \: (\md 4) \text{ and } \GCD(n,m)=k \text{ otherwise}\}.
\end{align}
As an example, consider $n=12$, in which case we have
\begin{align*}
  C_{12,1}&=\{1,5,7,11\},\\
  C_{12,2}&=\{2,4,8,10\},\\
  C_{12,3}&=\{3,9\},\\
  C_{12,6}&=\{6,12\}.\\
\end{align*}
Note in particular that sets $C_{12,1}$, $C_{12,2}$, $C_{12,3}$, and $C_{12,6}$ form a partition of $[12]$. More generally, the following lemma holds.
\begin{lemma}
  We have that $\displaystyle\coprod_{k\in D_n} C_{n,k}=[n]$, where $\coprod$ denotes disjoint union.
  \end{lemma}
  \begin{proof}
    First we show that for distinct $k,k'\in D_n$, we have that $C_{n,k}\cap C_{n,k'}=\emptyset$.
    Indeed, suppose to the contrary that there exists $m\in [n]$ belonging to $C_{n,k}\cap C_{n,k'}$.
    If  $k$ and $k'$ are such that  $\frac{n}{k}= \frac{n}{k'}= 2 \: (\md 4)$, then $\GCD(n,m)\in \{k,2k\}\cap \{k',2k'\}$. Note that $n$ is necessarily even in this case.
    The only way for $k$ and $k'$ to be distinct is if, say,  $\GCD(n,m)=2k=k'$.
    Since $\frac{n}{k}= 2 \: (\md 4)$, we know that  $\frac{n}{2k}$ is odd, which in turn means that $\frac{n}{k'}$ is odd. But this is absurd as, by definition, $k'\in D_n$ necessarily implies that $\frac{n}{k'}$ is even.

\medskip

An argument similar to the one just given also works in the case where we assume that $k$ and $k'$ are such that  $\frac{n}{k}= 2 \: (\md 4)$ and $ \frac{n}{k'}\neq 2 \: (\md 4)$.
    Finally, if  $k$ and $k'$ are such that  $\frac{n}{k}\neq 2 \: (\md 4)$ and $ \frac{n}{k'}\neq 2 \: (\md 4)$, then $\GCD(n,m)=k=k'$, a contradiction.
    Thus, we see that $C_{n,k}\cap C_{n,k'}=\emptyset$ for distinct $k,k'\in D_n$.

    \medskip

    To finish the proof, given $m\in [n]$, let $k\coloneqq \GCD(n,m)$.
    If $n$ is odd, or $n$ and $\frac{n}{k}$ are both even, then $k\in D_n$ and hence $m\in C_{n,k}$.
    Otherwise we are in the case where $n$ is even but $\frac{n}{k}$ is odd.
    It must be that $k$ is even as well, and therefore $\frac{k}{2}\in D_n$.
    Since $\GCD(n,m)=2\cdot \frac{k}{2}$ and $\frac{n}{k/2}=2 \: (\md 4)$, we have that $m\in C_{n,\frac{k}{2}}$.
  \end{proof}

\medskip

We are now ready for the classification.
\begin{theorem}
  \label{thm:iso_classes}
For $k\in D_n$, the representations $\tau_{n,c}$ are isomorphic for all $c\in C_{n,k}$.
Furthermore, for distinct $k,k' \in D_n$, we have that $\tau_{n,c} $ and $ \tau_{n,c'}$ are non-isomorphic for every $c\in C_{n,k}$ and $c'\in C_{n,k'}$.
\end{theorem}
\begin{proof}
  Pick $k\in D_n$.
  Note that $k\in C_{n,k}$.
  Consider $k' \in C_{n,k}$ distinct from $k$.
  Crucially, we are guaranteed that $\GCD(n,k')\in \{k,2k\}$.
   We first show that the characters $\chi_{{n,k}}$ and $\chi_{{n,k'}}$ agree on all conjugacy classes of $S_n$.
   We appeal to the character values defined by Theorem~\ref{thm:main_1} throughout our argument.
  Let $\pi\in S_n$ have cycle type $\lambda$, and write $d\coloneqq \GCD(\lambda)$.

\medskip

  \textsf{Case I:} Suppose $d$ is odd. To establish $\chi_{n,k}(\pi)$  and $\chi_{n,k'}(\pi)$ are equal, it suffices to show that $d|k\Leftrightarrow d|k'$. The forward direction is immediate, while the reverse implication follows easily from the fact that $\GCD(n,k')\in \{k,2k\}$, and we infer that $d|k$ as $d$ is odd.

\medskip

  \textsf{Case II:} Suppose $d$ and $\frac{n}{d}$ are both even. Once again, we need to show that $d|k\Leftrightarrow d|k'$, and we deal with the reverse direction.
  Assume $d|k'$.
  If $\GCD(n,k')=k$, then $d|k$ is immediate.
  We are left to deal with the case where $\frac{n}{k}= 2 \: (\md 4)$ and $\GCD(n,k')=2k$.
  Then we have that $\frac{n}{2k}$ is odd and $d|2k$.
  Now note that
  \begin{align}
    \frac{n}{2k}=\frac{n/d}{2k/d},
  \end{align}
  and since $\frac{n}{d}$ is even, it must be that $\frac{2k}{d}$ is even as well, from which it follows that $d|k$.

\medskip

\textsf{Case III:} Suppose $d$ is even but $\frac{n}{d}$ is odd.
We need to establish that $d|2k  \Leftrightarrow d|2k'$.
Again, we only need to deal with the reverse direction.
Suppose the stronger statement $d|k'$ holds.
 Since $\GCD(n,k')\in \{k,2k\}$, we have that $d|2k$.

Finally, suppose that $d\nmid k'$ but $d|2k'$. Then it must be that $\frac{2k'}{d}$ is odd.
If $\GCD(n,k')=k$, then by multiplying both sides by $2$, we conclude that $d|2k$.
Hence consider the case where $\frac{n}{k}=2 \: (\md 4)$ and $\GCD(n,k')=2k$.
Then we know that $\frac{n}{2k}$ is odd, and that $\GCD(\frac{2n}{d},\frac{2k'}{d})=\frac{4k}{d}$.
We conclude that $\frac{4k}{d}$ is odd as $\frac{2k'}{d}$ is odd.
Now note that $\frac{n}{2k}=\frac{2n/d}{4k/d}$
is odd, which is absurd as $\frac{2n}{d}$ is even while $\frac{4k}{d}$ is odd.
Thus, we see that the scenario $\frac{n}{k}=2 \: (\md 4)$ and $\GCD(n,k')=2k$ is impossible.

 \medskip

 At this stage, we know that for a fixed $k\in D_n$ the representations $\tau_{n,c}$ are isomorphic for all $c\in C_{n,k}$.
 In particular, they are isomorphic to $\tau_{n,k}$.
 To finish the proof, we show that $\tau_{n,k}$ and $ \tau_{n,k'}$ are nonisomorphic for distinct $k,k'\in D_n$ by finding a conjugacy class where they disagree.

 \medskip

 Without loss of generality, suppose $k>k'$.
 Let $\lambda\coloneqq (k^{\frac{n}{k}})$ and pick $\pi\in S_n$ with cycle type $\lambda$.
 We have $\GCD(\lambda)=k$, and therefore we see that $\chi_{{n,k}}(\pi)$ is nonzero, whereas $\chi_{{n,k'}}(\pi)$ is zero unless we are in the setting where $k$ is even, $\frac{n}{k}$ is odd, and $k|2k'$.
 This situation is impossible as $k\in D_n$ implies $\frac{n}{k}$ is even. This finishes the proof.
\end{proof}

\medskip

As an immediate consequence of Theorem~\ref{thm:iso_classes}, we have:
\begin{corollary}
  \label{cor:num_noniso_classes}
  There are $|D_n|$ many non-isomorphic representations among the $\tau_{n,c}$.
\end{corollary}
In view of Lemma~\ref{lem:num_D_n}, we have that $|D_n|$ is given by \cite[A320111]{oeis}.
Observe also the curious fact that the sequence $\{|D_n|\}_{n\geq 1}$ gives a multiplicative arithmetic function.
\begin{example}
  \emph{
  Consider $n=6$. Then $D_6=\{1,3\}$. Here are the power sum expansions for the two non-isomorphic representations amongst the $\tau_{6,c}$ for $c\in[6]$:
  \begin{align*}
    \mathrm{Frob}(\tau_{6,1})&=\frac{9}{5}p_{1^6} + \frac{9}{2}p_{21^4} + \frac{9}{4}p_{221^2} + \frac{1}{4}p_{222} + 2p_{31^3} + p_{321} + \frac{3}{4}p_{41^2} + \frac{1}{4}p_{42} + \frac{1}{5}p_{51},\\
    \mathrm{Frob}(\tau_{6,3})&=\frac{9}{5}p_{1^6} + \frac{9}{2}p_{21^4} + \frac{9}{4}p_{221^2} + \frac{1}{4}p_{222} + 2p_{31^3} + p_{321} + \frac{1}{2}p_{33} + \frac{3}{4}p_{41^2} + \frac{1}{4}p_{42} + \frac{1}{5}p_{51} + \frac{1}{2}p_{6}.
  \end{align*}
  The proof of Theorem~\ref{thm:iso_classes} predicts that $p_{33}$ appears with a nonzero coefficient in $\mathrm{Frob}(\tau_{6,3})$ but not in $\mathrm{Frob}(\tau_{6,1})$, as can be seen in the expansions.
  }
\end{example}


\section{Two special cases}
We now proceed to discuss the special cases $c=n$ (or equivalently $c=0$) and $c=1$.
We focus in particular on the number of orbits of our $S_n$ action, which is, in view of Burnside's lemma, equal to the multiplicity of the trivial representation.

\subsection{The case \texorpdfstring{$c=n$}{}}
Given a positive integer $n$, define an auxiliary function $f_n$ on the set of divisors of $n$ as follows:
\begin{align}
  f_n(d)=
  \left\lbrace \begin{array}{ll}
  1/2 & d \text{ even, } n/d \text{ odd,}\\
  1 & \text{otherwise.}
  \end{array}\right.
\end{align}
In terms of $f_n$, note that the character $\chi_{n,n}$ at the conjugacy class determined by the partition $\lambda\vdash n$ is $f_n(d)\cdot d^2n^{\ell(\lambda)-2}$ where $d\coloneqq \GCD(\lambda)$.
To establish a formula for the number of orbits, we need some more notation followed by a key lemma.
For a positive integer $m$, let $J_2(m)$ be the \emph{Jordan totient function} defined as
\begin{align}
  \label{eqn:jordan_totient_function}
  J_2(m)\coloneqq m^2\prod_{\text{prime }p|m}\left(1-\frac{1}{p^2}\right).
\end{align}
It is well known that $\sum_{d | m} J_2(d) = m^2$, in other words, $\sum_{d|m} \mu(m/d) d^2 = J_2(m)$.


\begin{lemma}
  \label{lem:jordan_totient}
  For fixed positive integer $e$ and $m$, let $n\coloneqq me$ and define
  \begin{align*}
    F(m,e)\coloneqq \sum_{d|m}\mu\left(\frac{m}{d}\right)f_{n}(d)d^2.
  \end{align*}
  Then we have that
  \begin{align*}
    F(m,e)=\left\lbrace
    \begin{array}{ll}
      J_2(m)  & e \text{ even or } m \text{ odd,}\\
      \frac{1}{3}J_2(m) & e\text{ odd and } m \text{ even.}
    \end{array}\right.
  \end{align*}
\end{lemma}
\begin{proof}
  Note that if $e$ is even, then $n/d$ is even, and hence $f_n(n/d)=1$.
  In this case we have
  \begin{align}
    \label{eqn:easy_case}
    F(m,e)=\sum_{d|m}\mu\left(\frac{m}{d}\right)d^2 = J_2(m).
  \end{align}

  \medskip

  Assume $e$ is odd.
  If $m$ is odd, then again we have that  $F(m,e)$ equals $J_2(m)$ from \eqref{eqn:easy_case}.
  Assume $v_2(m)=k \geq 1$. Then $m=2^{k}m'$, where $m'$ is odd.
  Note that if $d=2^{v_2(d)}d'$ is a divisor of $m$ such that $v_2(d)\leq k-2$, then
  \begin{align}
    \mu\left(\frac{m}{d}\right)=\mu\left(2^{k-v_2(d)}\frac{m'}{d'}\right),
  \end{align}
  where $\frac{m'}{d'}$ is odd.
  It follows that $\mu(m/d)=0$ as $m/d$ is not squarefree.

\medskip

Thus, we may assume that $d=2^{k-1}d'$ or $d=2^{k}d'$, where $d'|m'$.
In the former case, we have that $n/d$ is even, and in the latter case we have that $d$ is even while $n/d$ is odd. Thus we obtain
\begin{align}
  F(m,e)&=\sum_{d'|m'}\mu\left(\frac{2m'}{d'}\right)2^{2(k-1)}d'^2+\sum_{d'|m'}\mu\left(\frac{m'}{d'}\right)2^{2k-1}d'^2\nonumber \\&=
  2^{2k-2}\sum_{d'|m'}\mu\left(\frac{m'}{d'}\right)d'^2.
\end{align}
In arriving at the last equality, we made use of the fact that $\mu\left(\frac{2m'}{d'}\right)=-\mu\left(\frac{m'}{d'}\right)$. Therefore
$$F(m,e) = 2^{2k-2} J_2(m') = \frac 1 3 J_2(m),$$
which finishes the proof.


\end{proof}

\begin{theorem}
  \label{thm:orbits_c_0}
    Let $o_{n,n}$ denote the number of orbits of $\epf{n,n}$ under the action $\tau_{n,n}$. Then
    \begin{align*}
      o_{n,n}=\frac{1}{n^2}\sum_{e|n}\binom{2e-1}{e}F({n}/{e},e).
    \end{align*}
\end{theorem}
\begin{proof}
  By interpreting the number of orbits as the multiplicity of the trivial representation, we have that
  \begin{align}
    \label{eqn:initial_expression_0_orbit}
    o_{n,n}=\sum_{d|n}f_n(d)\sum_{
    \substack{\lambda\vdash n\\\GCD(\lambda)=d}
    }
    \frac{d^2n^{\ell(\lambda)-2}}{z_{\lambda}}.
  \end{align}
  Denote the inner sum in \eqref{eqn:initial_expression_0_orbit} by $a_{n,d}$.
  Note that if $\lambda=(\lambda_1,\dots,\lambda_{\ell})$ satisfies $\GCD(\lambda)=d$, then the partition $\tilde{\lambda}\coloneqq (\lambda_1/d,\dots,\lambda_{\ell}/d)$ satisfies $z_{\lambda}=d^{\ell(\lambda)}z_{\tilde{\lambda}}$ and $\GCD(\tilde{\lambda})=1$.
  Therefore, by scaling down the parts of $\lambda$ by $d$ in the sum defining $a_{n,d}$, we get that $a_{n,d}=a_{n/d,1}\eqqcolon a_{n/d}$.

  \medskip

  On the one hand, note that
  \begin{align}
    \label{eqn:rhs_1}
    \sum_{\lambda\vdash n}\frac{n^{\ell(\lambda)}}{z_{\lambda}}=\sum_{d|n}d^2a_{d}.
  \end{align}
  On the other hand, observe that $n^{\ell(\lambda)}$ is the specialization $p_\lambda(\underbrace{1,\ldots,1}_n,0,0,\ldots)$, and hence the relation $\sum_{\lambda} z_\lambda^{-1} p_\lambda = h_n$ implies
  \begin{align}\label{eqn:rhs_2}
  \sum_{\lambda\vdash n}\frac{n^{\ell(\lambda)}}{z_{\lambda}} = h_n(1,\ldots,1,0,0,\ldots) = \binom{2n-1}n.
  \end{align}
  Equating the right hand sides of equations ~\eqref{eqn:rhs_1} and \eqref{eqn:rhs_2} and then applying M\"{o}bius inversion to the resulting relation yields
  \begin{align}
    \label{eqn:nsquare_a_n}
    n^2a_{n}=\sum_{d|n}\mu\left(\frac{n}{d}\right)\binom{2d-1}{d},
  \end{align}

\medskip

Now rewrite \eqref{eqn:initial_expression_0_orbit} as
\begin{align}
  o_{n,n}=\sum_{d|n}f_n(d)a_{n/d},
\end{align}
and use \eqref{eqn:nsquare_a_n} to obtain
\begin{align}
  o_{n,n}&=\frac{1}{n^2}\sum_{d|n}f_n(d)d^2\sum_{e|\frac{n}{d}}\mu(n/ed)\binom{2e-1}{e}\nonumber\\
  &=\frac{1}{n^2}\sum_{e|n}\binom{2e-1}{e}\sum_{d|\frac{n}{e}}\mu(n/ed)f_n(d)d^2,
\end{align}
which gives $\frac{1}{n^2}\sum_{e|n}\binom{2e-1}{e}F({n}/{e},e)$ by the definition of $F$.
\end{proof}

Note that Lemma~\ref{lem:jordan_totient} tells us that the $F({n}/{e},e)$ appearing in the theorem is equal to $J_2(n/e)$, up to a potential factor of $1/3$.

\begin{example}
  \emph{
  Consider the case $\epf{3,3}$. Theorem~\ref{thm:main_1} gives the following power sum expansion for the Frobenius characteristic of the action $\tau_{3,3}$:
  \begin{align*}
    \mathrm{Frob}(\tau_{3,3})=p_3+\frac{p_{21}}{2}+\frac{p_{111}}{2}=2s_3+s_{111} = 3h_3-2h_{21}+h_{111}.
  \end{align*}
  We naturally expect a Schur-positive expansion for $\mathrm{Frob}(\tau_{3,3})$, though note that it is not $h$-positive.
  The Schur expansion also tells us that we have $2$ orbits under the action $\tau_{3,3}$. To see that we get the same quantity from Theorem~\ref{thm:orbits_c_0}, note that the possible values of $e$ are $1$ and $3$; the resulting $F(n/e,e)$ is equal to $8$ and $1$, respectively. Thus we get $o_{3,3}=\frac{1}{9}\left(1\cdot 8 + 10\cdot 1\right)=2$.
  }
\end{example}

\subsection{The case \texorpdfstring{$c=1$}{}.}
We now turn our attention to the case of the $S_n$ action on $\epf{n,1}$.
This is also the case which we believe (see Conjecture \ref{conj1}) to be pertinent from the viewpoint of the work of Berget-Rhoades \cite{BR14}.
Curiously, even though the character values turn out be simpler than those in the case $c=n$, unlike Theorem~\ref{thm:orbits_c_0}, we obtain a signed (yet still compact) expression for the number of orbits.

\medskip

\subsection*{Proof of Theorem \ref{thm2}}
As an immediate Corollary to Theorem~\ref{thm:main_1}, the character value $\chi_{n,1}(\pi)$, where $\pi\in S_n$ has cycle type $\lambda$, is given by
\begin{align}
  \label{eqn:char_vals_1}
\chi_{n,1}(\pi) = \begin{cases} n^{\ell-2} & d = 1 \\ 2n^{\ell-2} & d = 2, \: n = 2 \: (\md 4) \\ 0 & \text{otherwise.} \end{cases},
\end{align}
where $d\coloneqq \GCD(\lambda)$.

Let $o_{n,1}$ denote the number of orbits.
We borrow notation established in the proof of Theorem~\ref{thm:orbits_c_0}.
Using the description in \eqref{eqn:char_vals_1}, we have that
   the number of orbits $o_{n,1}$ under  $\tau_n$ on $\epf{n,1}$ satisfies
\begin{align}
o_{n,1} & = \left\lbrace \begin{array}{ll}
a_n+\frac{1}{2}a_{\frac{n}{2}} & \text{if } n = 2 \: (\md 4)\\
a_n & \text{otherwise}.
\end{array}\right.
\end{align}
From \eqref{eqn:nsquare_a_n}, we have an expression for $a_n$, which clearly agrees with the expression for $o_{n,1}$ in the statement of Theorem~\ref{thm2} in the case $n$ is odd, as $n+d$ is necessarily even.
Similarly, if $4|n$, then $\mu(n/d)$ is $0$ unless $d$ is an even divisor, and then $n+d$ is again even.
The case $n=2 \: (\md 4)$ needs some more manipulation.
In this case, we have
\begin{align}
  \label{eqn:o_n,1}
  o_{n,1}=\frac{1}{n^2}\sum_{d|n}\mu\left(\frac{n}{d}\right)\binom{2d-1}{d}+\frac{2}{n^2}\sum_{d|\frac{n}{2}}\mu\left(\frac{n}{2d}\right)\binom{2d-1}{d}
\end{align}
Note that $d|\frac{n}{2}$ is equivalent to stating that $d$ is an odd divisor of $n$.
Since $\frac{n}{2d}$ is odd for $d|\frac{n}{2}$, we have
\begin{align}
  \mu\left(\frac{n}{2d}\right)=-\mu\left(\frac{n}{d}\right),
\end{align}
which allows us to rewrite \eqref{eqn:o_n,1} as
\begin{align}
  o_{n,1}=\frac{1}{n^2}\sum_{\substack{d|n\\ d \text{ even}}}\mu\left(\frac{n}{d}\right)\binom{2d-1}{d}-\frac{1}{n^2}\sum_{\substack{d|n\\ d \text{ odd}}}\mu\left(\frac{n}{d}\right)\binom{2d-1}{d}.
\end{align}
Theorem~\ref{thm2} follows.

\begin{example}
  \emph{
  Consider the case $\epf{3,1}$. Theorem~\ref{thm:main_1} gives the following power sum expansion for the Frobenius characteristic of the action $\tau_{n,1}$:
  \begin{align*}
    \mathrm{Frob}(\tau_{3,1})=\frac{p_{21}}{2}+\frac{p_{111}}{2}=s_3+s_{21}=h_{21}.
  \end{align*}
  Again, we know that $\mathrm{Frob}(\tau_{3,1})$ is Schur-positive; in fact, in contrast to the case $c=n$, it is $h$-positive.
  The Schur expansion also tells us that we have $1$ orbit under the action $\tau_{3,1}$. To see that we get the same quantity from Theorem~\ref{thm2}, note that we have the possible $d$ being $1$ or $3$, and we get $o_{3,1}=\frac{1}{9}\left((-1)\cdot 1 + 1\cdot 10\right)=1$.}
\end{example}

\medskip

We conclude this section with a couple of remarks, the first of which concerns the multiplicity of another irreducible representation in the representation $\tau_{n,1}$.
{Let $[s_{\lambda}]\tau_{n,1}$ denote the multiplicity in $\tau_{n,1}$ of the irreducible representation of $S_n$ corresponding to $\lambda\vdash n$. In particular, $[s_n] \tau_{n,1} = o_{n,1}$.
	It follows easily that the multiplicity of the standard representation is
	\begin{align}\label{eqn:mult_standard}
	[s_{n-1,1}]\tau_{n,1}=\cat_{n-1} - o_{n,1}.
	\end{align}
	To see why this is true, recall on the one hand that $[s_{n-1}]\rho_{n-1}= \mathrm{Cat}_{n-1}$.
	On the other hand, since $\rho_{n-1}$ is the restriction of $\tau_{n-1}$, we have that $[s_{n-1}]\rho_{n-1}=[s_{n-1,1}]\tau_{n,1} + [s_{n}]\tau_{n,1}.$
	Since $[s_{n}]\tau_{n,1}=o_{n,1}$, the equality \eqref{eqn:mult_standard} follows.
}

\medskip

Our second remark is inspired by the grading in the Berget-Rhoades representation that stems from the area statistic for parking functions.
Given a parking function $\bfx=(x_1,\dots,x_n)\in \PF_n$, we define its \emph{area}, denoted by $\area{\bfx}$, as follows:
\begin{align}
\area{\bfx}=\binom{n}{2}-\sum_{1\leq i\leq n}x_i.
\end{align}
The area statistic was first studied by Kreweras \cite{Kre80} who related it to the inversion statistic on labeled trees.
If we now consider the $S_n$ module $\epf{n,c}$ where $c=\binom{n-1}{2} \: (\md n)$, then in the last coordinate of $\epf{n,c}$ we are recording the area statistic modulo $n$ of the parking function built from the preceding $n-1$ coordinates.
Observe that
\begin{align}
  \binom{n-1}{2}= \left\lbrace\begin{array}{ll}1 \: (\md n) & n \text{ odd}\\
  1+\frac{n}{2} \: (\md n) & n \text{ even.}
  \end{array}\right.
\end{align}
By using this fact, one can show that $\epf{n,c}$ is isomorphic to $\epf{n,1}$, which we conjecture to be the ungraded Berget-Rhoades representation.

\section{A generalization to certain families of rational parking functions}
\label{sec:generalization}
We broaden the scope of our results by applying our techniques to a subclass of the set of \emph{rational parking functions}.
These functions are a generalization of usual parking functions and their study is an active field of research in recent years \cite{ALW14, GM16, GMV16}.
Given the similarity in flavor to earlier arguments, we keep our exposition brief.

\medskip

	Consider coprime positive integers $a$ and $b$.
	Define an \emph{$(a,b)$-parking function} to be a sequence $(x_1,\dots,x_a)$ of nonnegative integers with the property that the weakly increasing arrangement $(z_1,\dots,z_a)$ satisfies $z_i \leq \frac{(i-1)b}{a}$.
	We denote the set of $(a,b)$-parking functions by $\PF_{a,b}$.
	As an example, consider the case where $a=3$ and $b=5$.
	The sequence $(1,0,3)$ is an element of $\PF_{3,5}$ as its weakly increasing arrangement $(0,1,3)$ satisfies the condition.
	On the other hand, it may be checked that $(2,0,3)\notin \PF_{3,5}$.
	It is clear that the set $\PF_{n,n+1}$ is the set $\PF_{n}$ from before.


%
%

We denote the natural action of $S_a$ on $\PF_{a,b}$ by $\rho_{a,b}$.
A generalization of Pollak's proof implies that the map from $\PF_{a,b}\to \Z_{b}^{a-1}$ given by mapping $$(x_1,\dots,x_a)\mapsto (x_2-x_1,\ldots,x_{a}-x_{a-1}),$$
where subtraction is performed modulo $b$, is a bijection.
This implies that $|\PF_{a,b}|=b^{a-1}$.
Furthermore, the number of orbits under the action $\rho_{a,b}$ is the rational Catalan number $\mathrm{Cat}_{a,b}$ defined to equal $\frac{1}{a+b}\binom{a+b}{b}$.
See \cite[Proposition 2]{ALW14}, \cite[Theorem 2.4.1]{Sulz17}, \cite[Theorem 3.1.1]{Thi15} for proofs establishing the aforementioned facts.

\medskip


\medskip

Mimicking our ideas from before, we construct a new set that is equinumerous with $\PF_{a,b}$.
For $c\in [b]$, define the set
	\[
	\epf{a,b,c} \coloneqq \{(x_1,\dots,x_{a+1}) \colon (x_1,\ldots,x_{a}) \in \PF_{a,b}, \: \sum_{1\leq i\leq a+1}x_i=c  \: (\md b) \}.
	\]
	As usual, we take $x_{a+1}$ to lie in $\{0,\dots,b-1\}$.
	Clearly, we have $|\epf{a,b,c}| = b^{a-1}$.

\medskip

In order to mimic our action from Section~\ref{sec:main_res}, we need to impose the constraint that $b|(a+1)$.
Henceforth, assume that this is indeed the case.
This given, we can construct an action $\tau_{a,b,c}$ of the symmetric group $S_{a+1}$ on $\epf{a,b,c}$.
	Take $\pi \in S_{a+1}$ and $(x_1,\ldots,x_{a+1}) \in \epf{a,b,c}$. Like before, $(x_{\pi_1},\ldots,x_{\pi_{a}})$ is not necessarily in $\PF_{a,b}$, and therefore $(x_{\pi_1},\ldots,x_{\pi_{a+1}})$ is not necessarily in $\epf{a,b,c}$.
	However, by the generalized Pollak's theorem, exactly one of the sequences $(y+x_{\pi_1},\ldots,y+x_{\pi_{a}})$ is in $\PF_{a,b}$, and therefore $(y+x_{\pi_1},\ldots,y+x_{\pi_{a+1}}) \in \epf{a,b,c}$.
	 This element is the action of $\pi$ on $(x_1,\ldots,x_{a+1})$.
	 The careful reader should note that we made use of the fact $b|(a+1)$ in obtaining an action.

Rather than repeating the analysis from before, we simply state our result for the case $c = 1$.
\begin{theorem}
	Take $a = kb-1$ for $b,k \in \N$. The map $\tau_{a,b,1}$ is an action of $S_{a+1}$ on $\epf{a,b,1}$ whose restriction to $S_{a}$ is isomorphic to $\rho_{a,b}$.
  Furthermore, the character $\chi_{\tau_{a,b,1}}$  can be computed as follows. Choose a permutation $\pi \in S_{a+1}$ with cycle type $(\lambda_1,\ldots,\lambda_\ell)$, and set $d \coloneqq \GCD(\lambda_1,\ldots,\lambda_\ell,b)$.
  Then
	$$\chi_{\tau_{a,b,1}(\pi)} = \begin{cases} b^{\ell-2} & d = 1 \\ 2b^{\ell-2} & d = 2, \: b = 2 \: (\md 4), \: k \text{ odd}\\ 0 & \text{otherwise} \end{cases}.$$
  Letting $o_{a,b,1}$ denote this number of orbits under $\tau_{a,b,1}$, we have the following equality:
  \begin{align*}
  o_{a,b,1}=\frac{1}{b^2} \sum_{d | b}(-1)^{k(b+d)} \mu(b/d) \binom{(k+1)d-1}{kd}.
\end{align*}
\end{theorem}

%

\section{Final remarks}
We remark briefly on a plausible approach to establishing Conjecture~\ref{conj1}.
One way to prove the conjecture would be to find an explicit action-preserving map between $\epf{n,1}$ and a particular basis of the space $V_n$.
The following table shows the construction (for a representative of each orbit) for $n = 3,4,5$.
Consider the case $n=3$ for instance. By its definition, $V_3$ would be spanned by elements of  $\{1,x_1-x_2,x_2-x_3,x_1-x_3\}$, and one can extract a basis from this, say $\{1,x_1-x_2,x_2-x_3\}$.
In fact, one can read from the table the following $S_3$-invariant basis of $V_3$:
$$\{1-2x_1+x_2+x_3,1-2x_2+x_1+x_3,1-2x_3+x_1+x_2\}.$$
The map
$$100 \mapsto 1-2x_1+x_2+x_3, \quad 010 \mapsto 1-2x_2+x_1+x_3, \quad 001 \mapsto 1-2x_3+x_1+x_2$$
commutes with the action.
We were not able to find an appropriate basis for $n \geq 6$, but we did check the conjecture (via character computations) for $n = 6$ as well.
Note further that $V_n$ is naturally graded by the number of edges of a slim graph.
We do not see a compatible grading in our $\epf{n}$.
$$\begin{array}{|c|c|c|} \hline
3 & 001 & 1-2x_3+x_1+x_2 \\ \hline
4 & 0003 & 1-3x_4+x_1+x_2+x_3 \\
  & 0012 & (x_4-x_1)(x_4-x_2)(1-2x_3+x_1+x_2) \\ \hline
5 & 00001 & 1-4x_5+x_1+x_2+x_3+x_4 \\
  & 00033 & (-3x_4+x_1+x_2+x_3)(-3x_5+x_1+x_2+x_3) \\
  & 01113 & (x_1-x_2)(x_1-x_3)(x_1-x_4) \\
  & 00114 & (x_5-x_3)(x_5-x_4)(-2x_3+x_1+x_2)(-2x_4+x_1+x_2) \\
  & 00123 & (x_5-x_1)(x_5-x_2)(x_5-x_3)(x_4-x_1)(x_4-x_2)(1-2x_3+x_1+x_2) \\ \hline
\end{array}$$

\medskip

Observe also the results of Berget and Rhoades are in a slightly more general setting\textemdash{}they consider spaces obtained by spans of polynomials attached to slim graphs of \emph{multigraphs} $K_{n}^{\ell,m}$ and study the $S_n$-action.
In this more general setup, usual parking functions are replaced by certain vector parking functions \cite{Yan15}.
 We emphasize that the generalization we consider in Section~\ref{sec:generalization} is different from the above-mentioned, even though rational parking functions are also vector parking functions.
Interestingly, usual parking functions are the only ones at the intersection of these two pictures.
\emph{Since Berget and Rhoades used an $S_n$-module coming from work of Postnikov and Shapiro, one is led to wonder if  there is an $(a,b)$-analogue in that context?}

\medskip

Our last remark concerns the number of orbits $o_{n,1}$ (respectively $o_{a,b,1}$) under the action $\tau_{n,1}$ (respectively $\tau_{a,b,1}$).
According to \cite[A131868]{oeis}, $no_{n,1}$ is equal to the number of $n$-element subsets of $\{1,\dots,2n-1\}$ that sum to $1$ modulo $n$.
We do not know how to establish this correspondence directly.
The numbers $o_{a,b,1}$ show up in a topological setting as Betti numbers as described in \cite[Section 5]{Ray18}.
Again the counting problem considered in the aforementioned article is different from ours.
We intend to explore some of these connections further.

\section*{Acknowledgements}

This material is based upon work supported by the Swedish Research
Council under
grant no. 2016-06596 while the authors were in residence at Institut
Mittag-Leffler in Djursholm, Sweden during Spring 2020.
We would like to thank the institute for its hospitality during our stay.
We would also like to thank Christos Athanasiadis, Darij Grinberg, Igor Pak, Jongwon Kim,  Marino Romero, and Robin Sulzgruber for helpful conversations and pointers to references.

%
%
%
%
%
%
%
\bibliographystyle{alpha}
\bibliography{Biblio_PS}

\end{document}